\documentclass[11pt]{amsart}
\usepackage{amssymb}
\usepackage{amsthm}
\usepackage{amsmath}
\usepackage{latexsym}
\usepackage{comment}
\usepackage{hyperref}
\usepackage{verbatim}
\usepackage{xypic}
\usepackage{graphicx}
\usepackage[utf8]{inputenc}
\usepackage[margin=1in]{geometry}
\newtheorem{thm}{Theorem}[section]
\newtheorem{prop}[thm]{Proposition}

\newtheorem{lemma}[thm]{Lemma}

\theoremstyle{definition}

\theoremstyle{definition} \newtheorem{rmk}[thm]{Remark}

\newcommand{\cc}{\mathbb{C}}

\newcommand{\qq}{\mathbb{Q}}
\newcommand{\zz}{\mathbb{Z}}
\newcommand{\ff}{\mathbb{F}}

\newcommand{\Gal}{\mathrm{Gal}}

\newcommand{\SL}{\mathrm{SL}}

\newcommand{\Aut}{\mathrm{Aut}}
\newcommand{\Spec}{\mathrm{Spec}}

\newcommand{\ab}{\mathrm{ab}}

\newcommand{\unr}{\mathrm{unr}}

\graphicspath{{pics/}}

\title{A note on $8$-division fields of elliptic curves}
\author{Jeffrey Yelton}

\begin{document}

\maketitle

\begin{abstract}

Let $K$ be a field of characteristic different from $2$ and let $E$ be an elliptic curve over $K$, defined either by an equation of the form $y^{2} = f(x)$ with degree $3$ or as the Jacobian of a curve defined by an equation of the form $y^{2} = f(x)$ with degree $4$.  We obtain generators over $K$ of the $8$-division field $K(E[8])$ of $E$ given as formulas in terms of the roots of the polynomial $f$, and we explicitly describe the action of a particular automorphism in $\Gal(K(E[8]) / K)$.

\end{abstract}

Let $K$ be any field of characteristic different from $2$, and let $E$ be an elliptic curve over $K$.  For any integer $N \geq 1$, we write $E[N]$ for the $N$-torsion subgroup of $E$ and $K(E[N])$ for the (finite algebraic) extension of $K$ obtained by adjoining the coordinates of the points in $E[N]$ to $K$.  Let $T_{2}(E)$ denote the $2$-adic Tate module of $E$; it is a free $\zz_{2}$-module of rank $2$ given by the inverse limit of the finite groups $E[2^{n}]$ with respect to the multiplication-by-$2$ map.  The absolute Galois group $G_{K} = \Gal(\bar{K} / K)$ of $K$ acts in a natural way on each free rank-$2$ $\zz / 2^{n}\zz$-module $E[2^{n}]$; we denote this action by $\bar{\rho}_{2^{n}} : G_{K} \to \Aut(E[2^{n}])$. This induces an action of $G_{K}$ on $T_{2}(E)$, which we denote by $\rho_{2} : G_{K} \to \Aut(T_{2}(E))$.

The purpose of this note is to provide formulas for generators of the $8$-division field $K(E[8])$ of an elliptic curve $E$ and to describe how a certain Galois element in $\Gal(K(E[8]) / K)$ acts on these generators.  We will consider the case where $E$ is given by a standard Weierstrass equation of the form $y^{2} = \prod_{i = 1}^{3} (x - \alpha_{i}) \in K[x]$ (the ``degree-$3$ case") and the case where $E$ is the Jacobian of the genus-$1$ curve given by an equation of the form $y^{2} = \prod_{i = 1}^{4} (x - \alpha_{i}) \in K[x]$ (the ``degree-$4$ case"), where in both cases the elements $\alpha_{i} \in \bar{K}$ are distinct.

For the statement of the main theorem and the rest of this article, we fix the following algebraic elements over $K$ (for ease of notation, we will treat indices $i$ as elements of $\zz / 3\zz$).  In the degree-$3$ case, for each $i \in \zz / 3\zz$, we choose an element $A_{i} \in \bar{K}$ whose square is $\alpha_{i + 1} - \alpha_{i + 2}$.  In the degree-$4$ case, for each $i \in \zz / 3\zz$, we choose an element $A_{i} \in \bar{K}$ whose square is $(\alpha_{i} - \alpha_{4})(\alpha_{i + 1} - \alpha_{i + 2})$.  One checks that in either case, we have the identity 
\begin{equation}\label{eq identity of A_{i}'s}
A_{1}^{2} + A_{2}^{2} + A_{3}^{2} = 0,
\end{equation}
which we will exploit below.

In the degree-$3$ case, it is well known that $K(E[2]) = K(\alpha_{1}, \alpha_{2}, \alpha_{3})$.  Menawhile, in the degree-$4$ case, the extension $K(E[2]) / K$ is generated by polynomials in the roots $\alpha_{1}, \alpha_{2}, \alpha_{3}, \alpha_{4}$ which are fixed by the group of permutations in $S_{4}$ that fix all partitions of the roots into $2$-element subsets.  This follows from a well-known description of the $2$-torsion points of the Jacobian of a hyperelliptic curve (see for instance the statement and proof of \cite[Corollary 2.11]{mumford1984tata}) which says that the points in $E[2]$ are parametrized by partitions of the set of roots $\{\alpha_{i}\}_{i = 1}^{4}$ into even-cardinality subsets, and that $G_{K}$ acts on $E[2]$ via the Galois action on these partitions determined by permutation of the $\alpha_{i}$'s.  In fact, it is clear (from examining, for instance, the solution to the ``generic" quartic equation via the resolvent cubic) that $K(E[2])$ coincides with $K(\gamma_{1}, \gamma_{2}, \gamma_{3})$, where $\gamma_{i} = (\alpha_{i + 1} + \alpha_{i + 2})(\alpha_{i} + \alpha_{4})$ for $i \in \zz / 3\zz$; note that $A_{i}^{2} = \gamma_{i + 1} - \gamma_{i + 2}$ for each $i$.  Thus, in either case, we have $A_{1}^{2}, A_{2}^{2}, A_{3}^{2} \in K(E[2])$.

Now for each $i \in \zz / 3\zz$, fix an element $B_{i} \in \bar{K}$ whose square is $A_{i}(A_{i + 1} + \zeta_{4}A_{i + 2})$.  Let $\zeta_{8} \in \bar{K}$ be a primitive $8$th root of unity, and let $\zeta_{4} = \zeta_{8}^{2}$, which is a primitive $4$th root of unity.  Our result is as follows.

\begin{thm} \label{thm 8-torsion}

a) We have $K(E[4]) = K(E[2], \zeta_{4}, A_{1}, A_{2}, A_{3})$ and $K(E[8]) = K(E[4], \zeta_{8}, B_{1}, B_{2}, B_{3})$.

b) If the scalar automorphism $-1 \in \Aut(E[8])$ lies in the image under $\bar{\rho}_{8}$ of some Galois element $\sigma \in G_{K}$, then $\sigma$ acts on $K(E[8])$ by fixing $K(E[2], \zeta_{8})$ and changing the sign of each generator $A_{i}, B_{i} \in \bar{K}$.

\end{thm}

\begin{rmk} \label{rmk dzb}

a) Rouse and Zureick-Brown have computed the full $2$-adic Galois images of all elliptic curves over $\qq$ in \cite{rouse2015elliptic}; in particular, their database can be used to find the image of $\bar{\rho}_{8}$ for any elliptic curve over $\qq$.  Our result allows one to view these mod-$8$ Galois images somewhat more explicitly.

b) For certain elliptic curves, it is possible to determine using various methods that the image of $\rho_{2}$ contains $\Gamma(8)$.  See \cite[Example 4.3]{yelton2017boundedness}, which shows this for elliptic curves in Legendre form whose Weierstrass roots satisfy certain arithmetic conditions; e.g. $y^{2} = x(x - 1)(x - 10)$.  One can then use our result to determine the full $2$-adic image in these cases.

\end{rmk}

The rest of this article is devoted to proving Theorem \ref{thm 8-torsion}.  We begin by justifying a simplifying assumption about the ground field $K$.  From now on, the superscript ``$S_{d}$" over a ring containing independent transcendental variables $\tilde{\alpha}_{1}, ... , \tilde{\alpha}_{d}$ indicates the subring of elements fixed under all permutations of the variables $\tilde{\alpha}_{i}$.

\begin{lemma} \label{lemma descent}

To prove Theorem \ref{thm 8-torsion} for the degree-$d$ case, it suffices to prove the statement when $K = \cc(\{\tilde{\alpha}_{i}\}_{i = 1}^{d})^{S_{d}}$ and the set of roots defining $E$ consists of the transcendental elements $\tilde{\alpha}_{i} \in \bar{K}$.

\end{lemma}

\begin{proof}

Assume that the statements of Theorem \ref{thm 8-torsion} are true in the degree-$d$ case when $K$ is $L := \cc(\{\tilde{\alpha}_{i}\}_{i = 1}^{d})^{S_{d}}$ and each root $\alpha_{i}$ is equal to $\tilde{\alpha}_{i}$.

\textit{Step 1:} We show that the statements are true for $K = k(\{\tilde{\alpha}_{i}\}_{i = 1}^{d})^{S_{d}}$, where $k$ is any subfield of $\cc$.  Due to the Galois equivariance of the Weil pairing, we have $\zeta_{4} \in K(E[4])$ and $\zeta_{8} \in K(E[8])$.  We will therefore assume that $\zeta_{8} \in k$, so that the image of $\Gal(\bar{K} / K)$ under $\rho_{2}$ modulo $4$ (resp. modulo $8$) is contained in the group $\SL(E[4])$ (resp. $\SL(E[8])$) of automorphisms of determinant $1$.  For any $n \geq 1$, write 
$$\phi_{2^{n}} : \Gal(K(E[2^{n}]) / K) \hookrightarrow \SL(E[2^{n}])$$
 for the obvious injection induced by $\rho_{2}$, and define $\phi_{2^{n}, L} : \Gal(L(E[2^{n}]) / L) \hookrightarrow \SL(E[2^{n}])$ analogously.  From the formulas given in Theorem \ref{thm 8-torsion} and elementary computations of the orders of the (finite) groups above for $n \in \{2, 3\}$, we see that $\phi_{4, L}$ and $\phi_{8, L}$ are isomorphisms.  Now let 
$$\theta_{2^{n}} : \Gal(L(E[2^{n}]) / L) \to \Gal(K(E[2^{n}]) / K)$$
 be the composition of the natural inclusion $\Gal(L(E[2^{n}]) / L) \hookrightarrow \Gal(L(E[2^{n}]) / K)$ with the natural restriction map $\Gal(L(E[2^{n}]) / K) \twoheadrightarrow \Gal(K(E[2^{n}]) / K)$.  Note that the automorphism in $\Gal(L(E[8]) / L(E[2]))$ which changes the sign of each generator given in Theorem \ref{thm 8-torsion} is sent by $\theta_{8}$ to the automorphism in $\Gal(K(E[8]) / K)$ which changes the sign of each of these generators.

It is clear that $\phi_{2^{n}, L} = \phi_{2^{n}} \circ \theta_{2^{n}}$.  It will therefore suffice to show that $\theta_{4}$ and $\theta_{8}$ are isomorphisms.  Indeed, they are injections due to the fact that $L(E[2^{n}])$ is the compositum of the subfields $K(E[2^{n}])$ and $\cc$ for each $n \geq 1$, and the fact that they are surjections in the case of $n \in \{2, 3\}$ follows immediately from the surjectivity of $\phi_{4, L}$ and $\phi_{8, L}$.

\textit{Step 2:} We show that the statements are true for $K = \ff_{p}(\{\tilde{\alpha}_{i}\}_{i = 1}^{d})^{S_{d}}$, where $p \neq 2$.  Let $E_{0}$ be the elliptic curve defined in the obvious way over $K_{0} := \qq(\{\tilde{\alpha}_{i}\}_{i = 1}^{d})^{S_{d}}$ in the degree-$d$ case.  By what was shown in Step 1, the statement of Theorem \ref{thm 8-torsion} is true for $E_{0}$.  It is easy to see that $E_{0}$ admits a model $\mathcal{E}$ over 
$$S := \Spec(\zz[\textstyle\frac{1}{2}, \{\tilde{\alpha}_{i}\}_{i = 1}^{d}, \{(\tilde{\alpha}_{i} - \tilde{\alpha}_{j})^{-1}\}_{1 \leq i < j \leq d}]^{S_{d}})$$
 which is an abelian scheme whose fiber over the prime $(p)$ is isomorphic to $E$.  For each $n \geq 1$, Proposition 20.7 of \cite{milne1986abelian} implies that the kernel of the multiplication-by-$2^{n}$ map on $\mathcal{E} \to S$, which we denote by $\mathcal{E}[2^{n}] \to S$, is a finite \'{e}tale group scheme over $S$.  Since the morphism $\mathcal{E}[2^{n}] \to S$ is finite, $\mathcal{E}[2^{n}]$ is an affine scheme; we write $\mathcal{O}_{S, 2^{n}} \supset \mathcal{O}_{S}$ for the minimal extension of scalars under which $\mathcal{E}[2^{n}]$ becomes constant.  Note that the ring $\zz[\frac{1}{2}, \{\tilde{\alpha}_{i}\}_{i = 1}^{d}, \{(\tilde{\alpha}_{i} - \tilde{\alpha}_{j})^{-1}\}_{1 \leq i < j \leq d}]$, along with all subrings of invariants under finite groups of automorphisms, is integrally closed.  It follows from the fact that $\mathcal{O}_{S}$ is integrally closed and from the finite \'{e}taleness of $\mathcal{O}_{S}[2^{n}]$ that $\mathcal{O}_{S, 2^{n}}$ is also integrally closed.  Moreover, the fraction field of $\mathcal{O}_{S, 2^{n}}$ coincides with $K_{0}(E_{0}[2^{n}])$ while the fraction field of $\mathcal{O}_{S, 2^{n}} / (p)$ coincides with $K(E[2^{n}])$ for each $n$.  We further observe that $\mathcal{O}_{S}[\{\tilde{\alpha}_{i}\}_{i = 1}^{3}]$ (resp. the subring of invariants in $\mathcal{O}_{S}[\{\tilde{\alpha}_{i}\}_{i = 1}^{4}]$ under the permutations which fix all partitions of $\{\tilde{\alpha}_{i}\}_{i = 1}^{4}$ into even-cardinality subsets) is integrally closed with fraction field equal to $K_{0}(E[2])$ and therefore coincides with $\mathcal{O}_{S, 2}$ in the degree-$3$ case (resp. in the degree-$4$ case).

We now need to show that $\mathcal{O}_{S, 2}[\zeta_{4}, A_{1}, A_{2}, A_{3}]$ and $\mathcal{O}_{S, 2}[\zeta_{8}, A_{1}, A_{2}, A_{3}, B_{1}, B_{2}, B_{3}]$ (where $A_{i}, B_{i}$ are the generators given in the statement of Theorem \ref{thm 8-torsion}) are integerally closed and therefore coincide with $\mathcal{O}_{S, 4}$ and $\mathcal{O}_{S, 8}$ respectively.  This claim follows from the elementary observation that extensions of integrally closed rings given by adjoining square roots of squarefree integral elements are integrally closed, provided that $2$ is invertible.  Moreover, the automorphism of $\mathcal{O}_{S, 8}$ which fixes $\mathcal{O}_{S, 2}$ and changes the signs of the $A_{i}$'s and $B_{i}$'s acts on the sections of the finite group scheme $\mathcal{E}[8] \to S$ as multiplication by $-1$, so we are done.

\textit{Step 3:} We show that the statement is true for $K = k(\{\tilde{\alpha}_{i}\}_{i = 1}^{d})^{S_{d}}$, where $k$ is any field of characteristic different from $2$.  Indeed, in the degree-$d$ case, the curve given by $y^{2} = \prod_{i = 1}^{d} (x - \tilde{\alpha}_{i})$ is clearly defined over $\mathfrak{F} := \mathfrak{f}(\{\tilde{\alpha}_{i}\}_{i = 1}^{d})^{S_{d}}$, where $\mathfrak{f}$ is the prime subfield of $k$, so $E$ is defined over $\mathfrak{F}$.  The statements are true over $\mathfrak{F}$ by what was shown in Steps 1 and 2, and the validity of the statements over $K$ follows from observing that $K(E[2^{n}]) = K \mathfrak{F}(E[2^{n}])$ for all $n \geq 1$.

\textit{Step 4:} We show that the statements are true when $K$ is any field of characteristic different from $2$.  Let $E_{0}$ be the elliptic curve defined in the obvious way over $K(\{\tilde{\alpha}_{i}\}_{i = 1}^{d})^{S_{d}}$.  Then $E_{0}$ admits a model over $\Spec(K[\{\tilde{\alpha}_{i}\}_{i = 1}^{d}, (\tilde{\alpha}_{i} - \tilde{\alpha}_{j})_{1 \leq i < j \leq d}]^{S_{d}})$, of which $E$ is the fiber over the closed point given by $\{\tilde{\alpha}_{i} = \alpha_{i}\}_{i = 1}^{d}$, and the claims follow by a similar argument to what was used for Step 2.

\end{proof}

In light of the above lemma, we assume from now on that the $\alpha_{i}$'s are independent transcendental variables over $\cc$ and that $K = \cc(\{\alpha_{i}\}_{i = 1}^{d})^{S_{d}}$ with $d = 3$ (resp. $d = 4$) in the degree-$3$ case (resp. the degree-$4$ case).  Note that again due to the equivariance of the Weil pairing, the image of $\rho_{2}$ is contained in the subgroup $\SL(T_{2}(E)) \subset \Aut(T_{2}(E))$ of automorphisms of determinant $1$.

We now recall the abstract definitions of the braid group and the pure braid group on $d$ strands for any integer $d \geq 1$.  The \textit{braid group on $d$ strands}, denoted $B_{d}$, is generated by elements $\beta_{1}, \beta_{2}, ... , \beta_{d - 1}$, with relations given by 
\begin{align}\label{braid relations}
\beta_{i}\beta_{j} &= \beta_{j}\beta_{i},& |i - j| \geq 2, \notag \\
\beta_{i}\beta_{i + 1}\beta_{i} &= \beta_{i + 1}\beta_{i}\beta_{i + 1}, & 1 \leq i \leq d - 2.
\end{align}
There is a surjective homomorphism of $B_{d}$ onto the symmetric group $S_{d}$ given by mapping each generator $\beta_{i}$ to the transposition $(i, i + 1) \in S_{d}$.  We define the \textit{pure braid group on $d$ strands}, denoted $P_{d} \lhd B_{d}$, to be the kernel of this surjection.  A set of generators $\{A_{i, j}\}_{1 \leq i < j \leq d}$ of $P_{d}$ with relations is given by \cite[Lemma 1.8.2]{birman1974braids}.  We write $\widehat{B}_{d}$ and $\widehat{P}_{d}$ for the profinite completions of $B_{d}$ and $P_{d}$ respectively.  Note that since $P_{d}$ and $B_{d}$ are residually finite, they inject into their profinite completions; moreover, since $P_{d} \lhd B_{d}$ has finite index, this injection induces an inclusion of profinite completions $\widehat{P}_{d} \lhd \widehat{B}_{d}$.

Let $G_{K}^{\unr}$ denote the Galois group of the maximal algebraic extension $K^{\unr}$ of $K$ which is unramified over the discriminant locus $\Delta$ of the affine scheme $\Spec (\cc[\{\alpha_{i}\}_{i = 1}^{d}]^{S_{d}})$, where $d = 3$ (resp. $d = 4$) in the degree-$3$ case (resp. the degree-$4$ case).  For any integer $n \geq 0$, we write $\Gamma(2^{n})$ for the level-$2^{n}$ congruence subgroup of $\SL(T_{2}(E))$ given by $\{\sigma \in \SL(T_{2}(E)) \ | \ \sigma \equiv 1 \ (\mathrm{mod} \ 2^{n})\}$.

\begin{lemma} \label{lemma key}

a) In the degree-$d$ case, we have an isomorphism $G_{K}^{\unr} \cong \widehat{B}_{d}$.

b) The map $\rho_{2} : G_{K} \to \SL(T_{2}(E))$ is surjective and factors through the obvious restriction map $G_{K} \twoheadrightarrow G_{K}^{\unr} \cong \widehat{B}_{d}$, inducing a surjection $\rho_{2}^{\unr} : \widehat{B}_{d} \twoheadrightarrow \SL(T_{2}(E))$.

c) For each $n \geq 0$, the algebraic extension $K(E[2^{n}]) / K$ is a subextension of $K^{\unr} / K$ and corresponds to the normal subgroup $(\rho_{2}^{\unr})^{-1}(\Gamma(2^{n})) \lhd \widehat{B}_{d}$.

d) The normal subgroup $(\rho_{2}^{\unr})^{-1}(\Gamma(2)) \lhd \widehat{B}_{d}$ coincides with $\widehat{P}_{3} \lhd \widehat{B}_{3}$ in the degree-$3$ case, and it coincides with a subgroup $H \lhd \widehat{B}_{4}$ which strictly contains $\widehat{P}_{4} \lhd \widehat{B}_{4}$ and which is isomorphic to $\widehat{P}_{3}$ in the degree-$4$ case.

\end{lemma}

\begin{proof}

Let $X_{d}$ denote the affine scheme $\Spec(\cc[\{\alpha_{i}\}_{i = 1}^{d}]^{S_{d}}) \setminus \Delta$, where $\Delta$ is the discriminant locus.  It is clear from definitions that $G_{K}^{\unr}$ can be identified with the \'{e}tale fundamental group of $X_{d}$.  Since $X_{d}$ is a complex scheme, it may also be viewed as a complex manifold, and so we may use Riemann's Existence Theorem (\cite{grothendieck224revetements}, Expos\'{e} XII, Corollaire 5.2) to identify its \'{e}tale fundamental group with the profinite completion of the fundamental group of the topological space $X_{d}$.  Now $X_{d}$ is the \textit{configuration space} of (unordered) $d$-element subsets of $\cc$, and it is well known that the fundamental group of $X_{d}$ is isomorphic to the braid group $B_{d}$.  Hence, $G_{K}^{\unr} \cong \widehat{B}_{d}$, and part (a) is proved.  It is also well known that the cover of $X_{d}$ corresponding to the normal subgroup $P_{d} \lhd B_{d}$ is given by the \textit{ordered configuration space} $Y_{d} := \Spec(\cc[\{\alpha_{i}\}_{i = 1}^{d}, \{(\alpha_{i} - \alpha_{j})^{-1}\}_{1 \leq i < j \leq d}])$ with its obvious map onto $X_{d}$.

To prove (b), we first note that $\rho_{2}$ is surjective because it is known that there exist elliptic curves with ``largest possible" $2$-adic Galois images; see also \cite[Corollary 1.2(b)]{yelton2015images}.  Now choose any prime $\mathfrak{p}$ of the coordinate ring of $X_{d}$ and note that $E$ has good reduction with respect to this prime.  It follows from the criterion of N\'{e}ron-Ogg-Shafarevich (\cite[Theorem 1]{serre1968good}) that the action $\rho_{2}$ is unramified with respect to $\mathfrak{p}$ and therefore factors through an algebraic extension of $K(\{\alpha_{i}\}_{i = 1}^{d})$ which is unramified over $\mathfrak{p}$.  The second claim of (b) follows.

Part (c) is immediate from the observation that the action $\bar{\rho}_{2^{n}} : G_{K} \to \Aut(E[2^{n}])$ is clearly the composition of $\rho_{2}$ with the quotient-by-$\Gamma(2^{n})$ map.

Finally, we investigate the subgroup $(\rho_{2}^{\unr})^{-1}(\Gamma(2)) \lhd \widehat{B}_{d}$.  In the degree-$3$ case, we get $(\rho_{2}^{\unr})^{-1}(\Gamma(2)) = \widehat{P}_{3} \lhd \widehat{B}_{3}$ from the fact that $K(E[2]) = K(\alpha_{1}, \alpha_{2}, \alpha_{3})$, which is the function field of the ordered configuration space $Y_{3}$ as defined above.  In the degree-$4$ case, we have seen that $K(E[2]) = K(\gamma_{1}, \gamma_{2}, \gamma_{3}) \subsetneq K(\alpha_{1}, \alpha_{2}, \alpha_{3}, \alpha_{4})$, which is the function field of $Y_{4}$.  Therefore, we have $H := (\rho_{2}^{\unr})^{-1}(\Gamma(2)) \supsetneq \widehat{P}_{4} \lhd \widehat{B}_{4}$.  It is easy to check that the $\gamma_{i}$'s are independent and transcendental over $\cc$, so that $K(E[2])$ and the function field of $Y_{3}$ are isomorphic as abstract $\cc$-algebras.  Thus, $H \cong \widehat{P}_{3}$, and (d) is proved.

\end{proof}

We now present several well-known group-theoretic facts which will be needed later.

\begin{lemma} \label{lemma exponent 2 abelian}

a) The centers of $B_{d}$ and $P_{d}$ are both generated by $\Sigma := (\beta_{1} \beta_{2} ... \beta_{d - 1})^{d} \in P_{d} \lhd B_{d}$; this element can be written as an ordered product of $1$st powers of all the generators $A_{i, j}$ in the presentation for $P_{d}$ given in \cite[Lemma 1.8.2]{birman1974braids}.

b) The abelianization of $P_{d}$ is isomorphic to $\zz^{d(d - 1) / 2}$.  More explicitly, it is freely generated by the images of the above generators $A_{i, j}$.

c) For each $n \geq 1$, the quotient $\Gamma(2^{n}) / \Gamma(2^{n + 1})$ is an elementary abelian group isomorphic to $(\zz / 2\zz)^{3}$.

\end{lemma}

\begin{proof}

The statement of (a) can be found in \cite[Corollary 1.8.4]{birman1974braids} and its proof.  Part (b) can be deduced directly from the presentation of $P_{d}$ mentioned above.  Part (c) can be seen easily from direct computations and is a special case of what is shown in the proof of \cite[Corollary 2.2]{sato2010abelianization}.

\end{proof}

It is now easy to determine the $4$-division field of $E$ in both cases and to describe how $G_{K}$ acts on it.  We note that in the degree-$3$ case, parts (a) and (b) are well known and can be deduced by straightforward calculations of order-$4$ points (for instance, in \cite[Example 2.2]{bekker2017divisibility}; see also \cite[Proposition 3.1]{yelton2015images}).

\begin{prop} \label{prop 4-torsion}

a) We have $K(E[4]) = K(E[2], A_{1}, A_{2}, A_{3})$.

b) Any Galois element $\sigma \in G_{K}$ with $\rho_{2}(\sigma) = -1 \in \SL(T_{2}(E))$ acts on $K(E[4])$ by fixing $K(E[2])$ and changing the signs of each generator $A_{i} \in \bar{K}$.

c) In the degree-$3$ case, the scalar automorphism $-1 \in \SL(T_{2}(E))$ is the image of the braid $\Sigma \in \widehat{P}_{3}$ under $\rho_{2}^{\unr}$.  In the degree-$4$ case, the scalar $-1 \in \SL(T_{2}(E))$ is the image of the braid $\Sigma \in \widehat{P}_{3} \cong H$, where $H \lhd \widehat{B}_{4}$ is the subgroup from the statement of Lemma \ref{lemma key}(c).

\end{prop}

\begin{proof}

Consider the composition of the restriction $\rho_{2}^{\unr} : (\rho_{2}^{\unr})^{-1}(\Gamma(2)) \twoheadrightarrow \Gamma(2)$ with the quotent map $\Gamma(2) \twoheadrightarrow \Gamma(2) / \Gamma(4)$.  Since $\Gamma(2) / \Gamma(4)$ is an abelian group of exponent $2$ by Lemma \ref{lemma exponent 2 abelian}(c), this composition must factor through the maximal exponent-$2$ abelian quotient $P_{3}^{\ab} / 2P_{3}^{\ab}$ of $(\rho_{2}^{\unr})^{-1}(\Gamma(2)) \cong \widehat{P}_{3}$.  We denote this induced surjection by $R : P_{3}^{\ab} / 2P_{3}^{\ab} \twoheadrightarrow \Gamma(2) / \Gamma(4)$.  It follows from parts (b) and (c) of Lemma \ref{lemma exponent 2 abelian} that both $P_{3}^{\ab} / 2P_{3}^{\ab}$ and $\Gamma(2) / \Gamma(4)$ are isomorphic to $(\zz / 2\zz)^{3}$, and so $R$ is an isomorphism.  Thus, $K(E[4])$ is the unique subextension of $K^{\unr} / K(E[2])$ with $\Gal(K(E[4]) / K(E[2])) \cong (\zz / 2\zz)^{3}$.  It is easy to check that in both cases, $K(E[2], A_{1}, A_{2}, A_{3})$ is such a subextension, and so $K(E[4]) = K(E[2], A_{1}, A_{2}, A_{3})$, proving (a).

Part (b) follows from checking that the automorphism of $K(E[4])$ defined by changing the signs of all the $A_{i}$'s is the only nontrivial automorphism lying in the center of $\Gal(K(E[2], A_{1}, A_{2}, A_{3}) / K)$.

Now it follows from (a) and (b) of Lemma \ref{lemma exponent 2 abelian} that $\Sigma$ has nontrivial image in $P_{3}^{\ab} / 2P_{3}^{\ab}$.  It therefore has nontrivial image in $\Gamma(2) / \Gamma(4)$, so $\rho_{2}^{\unr}(\Sigma)$ is a nontrivial element of $\SL(T_{2}(E))$.  We know from Lemma \ref{lemma exponent 2 abelian}(a) that $\Sigma$ lies in the center of $\widehat{P}_{3}$.  It follows from Lemma \ref{lemma key}(d) that $\rho_{2}^{\unr}$ restricted to $(\rho_{2}^{\unr})^{-1}(\Gamma(2)) \cong \widehat{P}_{3}$ is surjective onto $\Gamma(2) \lhd \SL(T_{2}(E))$, so it takes the center of $\widehat{P}_{3}$ to the center of $\Gamma(2)$, which is $\{\pm 1\}$.  We therefore get $\rho_{2}(\sigma) = -1 \in \Gamma(2)$, which is the statement of (c).

\end{proof}

We now want to find generators for the extension $K(E[8]) / K(E[2])$.  In order to do so, we will first prove that $\widehat{P}_{3}$ has a unique quotient isomorphic to $\Gamma(2) / \Gamma(8)$ (Lemma \ref{lemma unique quotient} below), and then we will show that the extension of $K(E[2])$ given in the statement of Theorem \ref{thm 8-torsion} has Galois group isomorphic to $\Gamma(2) / \Gamma(8)$ (Lemma \ref{lemma Galois group} below).  For the following, we note that after fixing a basis of the free rank-$2$ $\zz_{2}$-module $T_{2}(E)$, we may consider $\SL(T_{2}(E))$ as the matrix group $\SL_{2}(\zz_{2})$.  Moreover, by applying a suitable form of the Strong Approximation Theorem (see for instance Theorem 7.12 of \cite{platonov1993algebraic}), we have $\Gamma(2) / \Gamma(8) \cong (\Gamma(2) \cap \SL_{2}(\zz)) / (\Gamma(8) \cap \SL_{2}(\zz))$.  In light of this, in the proofs of the next two lemmas, we use the symbols $\Gamma(2)$ and $\Gamma(8)$ to denote principal congruence subgroups of $\SL_{2}(\zz)$ rather than of $\SL_{2}(\zz_{2}) \cong \SL(T_{2}(E))$.

\begin{lemma} \label{lemma group structure}

The group $\Gamma(2)$ decomposes into a direct product of the scalar subgroup $\{\pm 1\}$ with another subgroup $\Gamma(2)'$.  The quotient $\Gamma(2)' / \Gamma(8)$ can be presented as 
\begin{equation}\label{presentation of Gamma(2) / Gamma(8)}
\langle \sigma, \tau \ | \ \sigma^{4} = \tau^{4} = [\sigma^{2}, \tau] = [\sigma, \tau^{2}] = [\sigma, \tau]^{2} = [[\sigma, \tau], \sigma] = [[\sigma, \tau], \tau] = 1 \rangle.
\end{equation}

\end{lemma}

\begin{proof}

Let $\Gamma(2)'$ be the subgroup consisting of matrices in $\Gamma(2)$ whose diagonal entries are equivalent to $1$ modulo $4$.  Then it is straightforward to check that $\Gamma(2) = \{\pm 1\} \times \Gamma(2)'$.  We note that by Lemma \ref{lemma exponent 2 abelian}(c), the order of $\Gamma(2) / \Gamma(8)$ is $64$, and so since $-1 \notin \Gamma(8)$, the order of $\Gamma(2)' / \Gamma(8)$ is $32$.

Let $\sigma$ (resp. $\tau$) be the image of $\tilde{\sigma} := \begin{bmatrix} 1 & -2 \\ 0 & 1 \end{bmatrix}$ (resp. $\tilde{\tau} := \begin{bmatrix} 1 & 0 \\ 2 & 1 \end{bmatrix}$) in $\Gamma(2)' / \Gamma(8)$.  It is well known that $\tilde{\sigma}$ and $\tilde{\tau}$ generate $\Gamma(2)' \cap \SL_{2}(\zz)$ (see, for instance, Proposition A.1 of \cite{mumford1983tata}), so $\sigma$ and $\tau$ generate $\Gamma(2)' / \Gamma(8)$.  It is then straightforward to check that the relations given in (\ref{presentation of Gamma(2) / Gamma(8)}) hold.  To show that these relations determine the group $\Gamma(2)' / \Gamma(8)$, one checks that the only nontrivial element of the commutator subgroup of the group given by (\ref{presentation of Gamma(2) / Gamma(8)}) has order $2$ and that the quotient by the commutator subgroup is isomorphic to $\zz / 4 \zz \times \zz / 4\zz$; therefore, the group has order $32$.  Since $\Gamma(2)' / \Gamma(8)$ also has order $32$, it must be fully determined by the relations in (\ref{presentation of Gamma(2) / Gamma(8)}).

\end{proof}

\begin{lemma} \label{lemma unique quotient}

The only normal subgroup of $\widehat{P}_{3}$ which induces a quotient isomorphic to $\Gamma(2) / \Gamma(8)$ is $(\rho_{2}^{\unr})^{-1}(\Gamma(8)) \lhd (\rho_{2}^{\unr})^{-1}(\Gamma(2)) \cong \widehat{P}_{3}$.

\end{lemma}

\begin{proof}

Since $\widehat{P}_{3}$ and $P_{3}$ have the same finite quotients, it suffices to show that the only normal subgroup of $P_{3}$ inducing a quotient isomorphic to $\Gamma(2) / \Gamma(8)$ coincides with $(\rho_{2}^{\unr})^{-1}(\Gamma(8)) \cap P_{3}$.  Let $N \lhd P_{3}$ be a normal subgroup whose corresponding quotient is isomorphic to $\Gamma(2) / \Gamma(8)$.  By Lemma \ref{lemma exponent 2 abelian}(a), the braid $\Sigma$ generates the center of $P_{3}$; therefore, its image modulo $N$ must lie in the center of $P_{3} / N \cong \Gamma(2) / \Gamma(8)$.  It can easily be deduced from Lemma \ref{lemma group structure} that the center of $\Gamma(2) / \Gamma(8)$ is an elementary abelian $2$-group, so the image of $\Sigma$ modulo $N$ must have order dividing $2$.  We claim that $\Sigma \notin N$.  Indeed, if $\Sigma \in N$, then $P_{3} / N$ could be generated by the images of only $2$ of the generators of $P_{3}$ given above.  But it is clear from Lemma \ref{lemma group structure} that $\Gamma(2) / \Gamma(8) = \{\pm 1\} \times \Gamma(2)' / \Gamma(8)$ cannot be generated by only $2$ elements, a contradiction.  Therefore, the image of $\Sigma$ modulo $N$ has order $2$, so $\Sigma^{2} \in N$ and the quotient map factors through $P_{3} / \langle \Sigma^{2} \rangle$.  But the discussion in \cite[\S3.6.4]{farb2011primer} shows that $P_{3} / \langle \Sigma^{2} \rangle \cong \Gamma(2) \lhd \SL_{2}(\zz)$.  We claim that in fact, the kernel of $\rho_{2}^{\unr}$ coincides with $\langle \Sigma^{2} \rangle$, so that the quotient-by-$N$ map factors through $\rho_{2}^{\unr} : P_{3} \twoheadrightarrow \Gamma(2)$.  Since $\rho_{2}^{\unr}(\Sigma) = -1 \in \SL_{2}(\zz_{2})$ by (b) and (c) of Proposition \ref{prop 4-torsion}, we know that the kernel of $\rho_{2}^{\unr}$ contains $\langle \Sigma^{2} \rangle$, and to prove the claim we need to show that $\Gamma(2)$ has no proper quotient isomorphic to itself.  But this follows from the fact that $\Gamma(2)$ is finitely generated and is residually finite, so the claim holds.  Therefore, to prove the statement of the lemma, it suffices to show that $\Gamma(8)$ is the only normal subgroup of $\Gamma(2)$ which induces a quotient isomorphic to $\Gamma(2) / \Gamma(8)$.

Any surjection $\Gamma(2) \twoheadrightarrow \Gamma(2) / \Gamma(8)$ takes $-1 \in \Gamma(2)$ to a nontrivial element $\mu \in \Gamma(2) / \Gamma(8)$ and takes $\Gamma(2)'$ to some proper subgroup of $\Gamma(2) / \Gamma(8)$ not containing $\mu$, since $\Gamma(2)'$ can be generated by only $2$ elements while $\Gamma(2) / \Gamma(8)$ cannot.  Therefore, such a surjection takes $\Gamma(2)'$ to a subgroup of $\Gamma(2) / \Gamma(8)$ isomorphic to $\Gamma(2)' / \Gamma(8)$.  So in fact it suffices to show that $\Gamma(8)$ is the only normal subgroup of $\Gamma(2)'$ which induces a quotient isomorphic to $\Gamma(2)' / \Gamma(8)$.

Let $N' \lhd \Gamma(2)'$ be a normal subgroup such that $\Gamma(2)' / N' \cong \Gamma(2)' / \Gamma(8)$.  Let $\tilde{\sigma}$ and $\tilde{\tau}$ be the matrices given in the proof of Lemma \ref{lemma group structure}, and let $\phi_{N'} : \Gamma(2)' \twoheadrightarrow \Gamma(2)' / N'$ be the obvious quotient map.  One checks from the presentation given in the statement of Lemma \ref{lemma group structure} that each element of $\Gamma(2)' / N'$ has order dividing $4$; that each square element lies in the center; and that each commutator has order dividing $2$ and lies in the center.  It follows that $\phi_{N'}(\tilde{\sigma}^{4}) = \phi_{N'}(\tilde{\tau}^{4}) = \phi_{N'}([\tilde{\sigma}^{2}, \tilde{\tau}]) = \phi_{N'}([\tilde{\sigma}, \tilde{\tau}^{2}]) = \phi_{N'}([\tilde{\sigma}, \tilde{\tau}]^{2}) = \phi_{N'}([[\tilde{\sigma}, \tilde{\tau}], \tilde{\sigma}]) = \phi_{N'}([[\tilde{\sigma}, \tilde{\tau}], \tilde{\tau}]) = 1$.  Thus, $N'$ contains the subgroup normally generated by $\{\tilde{\sigma}^{4}, \tilde{\tau}^{4}, [\tilde{\sigma}^{2}, \tilde{\tau}], [\tilde{\sigma}, \tilde{\tau}^{2}], [\tilde{\sigma}, \tilde{\tau}]^{2}, [[\tilde{\sigma}, \tilde{\tau}], \tilde{\sigma}], [[\tilde{\sigma}, \tilde{\tau}], \tilde{\tau}]\}$.  But Lemma \ref{lemma group structure} implies that $\Gamma(8) \lhd \Gamma(2)'$ is normally generated by this subset, so $\Gamma(8) \unlhd N'$.  Since $\Gamma(2)' / N'$ and $\Gamma(2)' / \Gamma(8)$ have the same (finite) order, we have $N' = \Gamma(8)$, as desired.

\end{proof}

\begin{lemma} \label{lemma Galois group}

The Galois group $\Gal(K(E[2], A_{1}, A_{2}, A_{3}, B_{1}, B_{2}, B_{3}) / K(E[2]))$ is isomorphic to $\Gamma(2) / \Gamma(8)$.

\end{lemma}

\begin{proof}

Let $K' = K(E[2], A_{1}, A_{2}, A_{3}, B_{1}, B_{2}, B_{3})$.  Clearly $K'$ is generated over $K(E[4])$ by square roots of three elements which are independent in $K(E[4])^{\times} / (K(E[4])^{\times})^{2}$, and thus, $[K' : K(E[4])] = 8$.  Therefore, since $[K(E[4]) : K(E[2])] = 8$, we have $[K' : K(E[2])] = 64$.

Using the relation (\ref{eq identity of A_{i}'s}), for each $i$, we compute 
\begin{equation}\label{eq identity of A_{i}'s2}
(A_{i}(A_{i + 1} + \zeta_{4}A_{i + 2}))(A_{i}(A_{i + 1} - \zeta_{4}A_{i + 2})) = -A_{i}^{4}.
\end{equation}
In light of this, for $i \in \zz / 3\zz$, we define $B_{i}'$ to be the element of $K'$ such that $B_{i}'^{2} = A_{i}(A_{i + 1} - \zeta_{4}A_{i + 2})$ and $B_{i}B_{i}' = \zeta_{4}A_{i}^{2} \in K(E[2])$.  Define $\sigma \in \Gal(K' / K(E[2]))$ as the automorphism which acts by 
$$\sigma : (A_{1}, A_{2}, A_{3}, B_{1}, B_{2}, B_{3}) \mapsto (A_{1}, A_{2}, -A_{3}, B_{1}', \zeta_{4}B_{2}', \zeta_{4}B_{3}),$$
and let $\tau \in \Gal(K' / K(E[2]))$ be the automorphism which acts by 
$$\tau : (A_{1}, A_{2}, A_{3}, B_{1}, B_{2}, B_{3}) \mapsto (-A_{1}, A_{2}, A_{3}, \zeta_{4}B_{1}, B_{2}', \zeta_{4}B_{3}').$$
Note that $\sigma^{2}$ and $\tau^{2}$ both act trivially on $K(E[4])$ while sending $(B_{1}, B_{2}, B_{3})$ to $(B_{1}, B_{2}, -B_{3})$ and to $(-B_{1}, B_{2}, B_{3})$ respectively; it is now easy to check that $\sigma^{2}$ (resp. $\tau^{2}$) has order $2$ and commutes with $\tau$ (resp. $\sigma$).  One also verifies that $[\sigma, \tau]$ acts trivially on $K(E[4])$ and sends $(B_{1}, B_{2}, B_{3})$ to $(-B_{1}, -B_{2}, -B_{3})$, and that this automorphism also commutes with both $\sigma$ and $\tau$.  Thus, $\sigma$ and $\tau$ satisfy all of the relations given in (\ref{presentation of Gamma(2) / Gamma(8)}).  Moreover, $\sigma$ and $\tau$ each have order $4$, while $[\sigma, \tau]$ has order $2$.  It is elementary to verify that this implies that $\langle \sigma, \tau \rangle$ has order $32$, which is the order of $\Gamma(2)' / \Gamma(8)$; therefore $\langle \sigma, \tau \rangle \cong \Gamma(2)' / \Gamma(8)$.  Note also that $\langle \sigma, \tau \rangle$ fixes $A_{2}$, whose orbit under $\Gal(\bar{K} / K(E[2]))$ has cardinality $2$, so if $\mu$ is any automorphism in $\Gal(K' / K(E[2]))$ which does not fix $A_{2}$, then $\langle \sigma, \tau, \mu \rangle$ has order $64$ and must be all of $\Gal(K' / K(E[2]))$.  Let $\mu$ be the automorphism that acts by changing the sign of all $A_{i}$'s and all $B_{i}$'s.  Then $\mu$ commutes with $\sigma$ and $\tau$, and 
\begin{equation}\label{Gal(L' / L_{1}) is Gamma(2) / Gamma(8)}
\Gal(K' / K(E[2])) = \langle \sigma, \tau \rangle \times \langle \mu \rangle \cong \Gamma(2)' / \Gamma(8) \times \{\pm 1\} \cong \Gamma(2) / \Gamma(8).
\end{equation}

\end{proof}

The next two propositions (Propositions \ref{prop 8-torsion generators} and \ref{prop action of -1} below) imply Theorem \ref{thm 8-torsion}.

\begin{prop} \label{prop 8-torsion generators}

We have $K(E[8]) = K(E[2], A_{1}, A_{2}, A_{3}, B_{1}, B_{2}, B_{3})$.

\end{prop}

\begin{proof}

As before, write $K'$ for $K(E[2], A_{1}, A_{2}, A_{3}, B_{1}, B_{2}, B_{3})$.  It is straighforward to check by computing norms that the field extension $K' / K(E[2])$ obtained by adjoining the $A_{i}$'s and $B_{i}$'s is unramified away from the discriminant locus (the union of the primes $(\alpha_{i} - \alpha_{j})$) and thus, $K'$ is a subextension of $K^{\unr} / K$.  Lemma \ref{lemma key}(d) tells us that $\Gal(K^{\unr} / K(E[2])) \cong \widehat{P}_{3}$, so the subextension $K'$ corresponds to some normal subgroup of $\widehat{P}_{3}$ inducing a quotient isomorphic to $\Gal(K' / K(E[2])) \cong \Gamma(2) / \Gamma(8)$.  Lemma \ref{lemma unique quotient} then implies that this normal subgroup of $\widehat{P}_{3}$ is the one corresponding to $(\rho_{2}^{\unr})^{-1}(\Gamma(8)) \lhd \widehat{P}_{3}$.  But Lemma \ref{lemma key}(c) says that the subextension corresponding to $(\rho_{2}^{\unr})^{-1}(\Gamma(8))$ is $K(E[8])$.  Therefore, $K' = K(E[8])$, as desired.

\end{proof}

\begin{rmk} \label{rmk computations}

We may now use Proposition \ref{prop 8-torsion generators} to compute several elements that lie in $K(E[8])$.

a) We first compute, for $i \in \zz / 3\zz$ (and with $B_{i}'$ defined as in the proof of Lemma \ref{lemma Galois group}), that 
\begin{equation} (B_{i} \pm B_{i}')^{2} = 2A_{i}A_{i + 1} \pm 2\zeta_{4}A_{i}^{2}; \ (B_{i} \pm \zeta_{4}B_{i}')^{2} = 2\zeta_{4}A_{i}A_{i + 2} \mp 2A_{i}^{2}. \end{equation}
Therefore, for each $i$, we have (up to sign changes) 
$$\sqrt{-A_{i}^{2} \pm \zeta_{4}A_{i}A_{i + 1}} = (1 \mp \zeta_{4})^{-1}(B_{i} \pm B_{i}'), \ \sqrt{\zeta_{4}A_{i}A_{i + 2} \pm A_{i}^{2}} = (\zeta_{8} + \zeta_{8}^{-1})^{-1}(B_{i} \mp \zeta_{4}B_{i}') \in K(E[8]).$$

b) We similarly compute, for $i \in \zz / 3\zz$, that  
$$2^{-1}\zeta_{4}(B_{i} - B_{i}')^{2}B_{i + 2}^{2} / (A_{i} + \zeta_{4}A_{i + 1})^{2} = (A_{i}(A_{i} + \zeta_{4}A_{i + 1}))(A_{i + 2}(A_{i} + \zeta_{4}A_{i + 1})) / (A_{i} + \zeta_{4}A_{i + 1})^{2}$$
\begin{equation}\label{eq sqrt of product of A's} = A_{i}A_{i + 2}(A_{i} + \zeta_{4}A_{i + 1})^{2} / (A_{i} + \zeta_{4}A_{i + 1})^{2} = A_{i}A_{i + 2}. \end{equation}
Therefore, for each $i$, we have $\pm\sqrt{A_{i}A_{i + 2}} = \pm(1 - \zeta_{4})^{-1}(B_{i} - B_{i}')B_{i + 2} / (A_{i} + \zeta_{4}A_{i + 1}) \in K(E[8])$.

\end{rmk}

\begin{prop} \label{prop action of -1}

Any Galois element $\sigma \in G_{K}$ with $\rho_{2}(\sigma) = -1 \in \SL(T_{2}(E))$ acts on $K(E[8])$ by changing the sign of each of the generators $A_{i}, B_{i} \in \bar{K}$.

\end{prop}

\begin{proof}

Let $\sigma \in G_{K}$ be an automorphism which acts on $K(E[8])$ by changing the sign of each generator $A_{i}, B_{i} \in \bar{K}$.  Then it follows from Proposition \ref{prop 4-torsion}(b) that $\bar{\rho}_{4}(\sigma)$ is the scalar $-1 \in \SL(E[4])$.  Moreover, we observe that the restriction of $\sigma$ to $K(E[8])$ lies in the center of $\Gal(K(E[8]) / K)$, so $\bar{\rho}_{8}(\sigma)$ is a scalar automorphism in $\SL(E[8])$, either $-1$ or $3$.  In order to determine which scalar it is, we first treat the degree-$3$ case and compute a point of order $8$ in $E(\bar{K})$.  To simplify computations, we instead work with the elliptic curve $E'$ defined by $y^{2} = x(x - (\alpha_{2} - \alpha_{1}))(x - (\alpha_{3} - \alpha_{1}))$, which is isomorphic to $E$ over $K(E[2])$ via the morphism $(x, y) \mapsto (x - \alpha_{1}, y)$.  (We note that replacing the roots $\alpha_{i}$ with the new roots $\alpha_{i}' := \alpha_{i} - \alpha_{1}$ in the formulas for $A_{i}, B_{i} \in \bar{K}$ does not change the elements $A_{i}, B_{i} \in \bar{K}$.)

Given a point $(x_{0}, y_{0}) \in E(\bar{K})$, in \cite[\S2]{bekker2017divisibility}, Bekker and Zarhin describe an algorithm to find a point $P \in E(\bar{K})$ with $2P = (x_{0}, y_{0})$.  In order to find such a point, one chooses elements $r_{1}, r_{2}, r_{3} \in \bar{K}$ with $r_{i}^{2} = x_{0} - \alpha_{i}$ for $i = 1, 2, 3$ and with $r_{1}r_{2}r_{3} = -y_{0}$.  Then 
$$P := (x_{0} + (r_{1}r_{2} + r_{2}r_{3} + r_{3}r_{1}), -y_{0} + (r_{1} + r_{2} + r_{3})(r_{1}r_{2} + r_{2}r_{3} + r_{3}r_{1}))$$
 satisfies $2P = (x_{0}, y_{0})$.  Following this algorithm, we get a point $P$ of order $4$ with $2P = (0, 0)$ given by $P = (\zeta_{4}A_{2}A_{3}, \zeta_{4}A_{2}A_{3}(A_{2} + \zeta_{4}A_{3}))$.  Similarly, we get a point $Q$ of order $8$ with $2Q = P$ given by 
$$Q = (\zeta_{4}A_{2}A_{3} + (r_{1}r_{2} + r_{2}r_{3} + r_{3}r_{1}), -\zeta_{4}A_{2}A_{3}(A_{2} + \zeta_{4}A_{3}) + (r_{1} + r_{2} + r_{3})(r_{1}r_{2} + r_{2}r_{3} + r_{3}r_{1})),$$
 where $r_{1}, r_{2}, r_{3} \in \bar{K}$ are elements satisfying 
$$r_{1}^{2} = \zeta_{4}A_{2}A_{3} - \alpha_{1}' = \zeta_{4}A_{2}A_{3}; \ r_{2}^{2} = \zeta_{4}A_{2}A_{3} - \alpha_{2}' = \zeta_{4}A_{2}A_{3} + A_{3}^{2}; \ r_{3}^{2} = \zeta_{4}A_{2}A_{3} - \alpha_{3}' = \zeta_{4}A_{2}A_{3} - A_{2}^{2};$$
 and $r_{1}r_{2}r_{3} = -\zeta_{4}A_{2}A_{3}(A_{2} + \zeta_{4}A_{3})$.  Using the formulas computed in Remark \ref{rmk computations}, we see that one may choose 
$$r_{1} \in \{\pm(\zeta_{8} - \zeta_{8}^{-1})^{-1}(B_{3} - B_{3}')B_{2} / (A_{3} + \zeta_{4}A_{1})\}, \ r_{2} \in \{\pm(\zeta_{8} + \zeta_{8}^{-1})^{-1}(B_{3} - \zeta_{4}B_{3}')\}, \ r_{3} \in \{\pm(1 - \zeta_{4})^{-1}(B_{2} + B_{2}')\}$$
 with $r_{1}r_{2}r_{3}$ as specified above.  It follows that the $x$-coordinate (resp. the $y$-coordinate) of $Q$ can be written as a quotient of homogeneous polynomial functions in the $A_{i}$'s, $B_{i}$'s, and $B_{i}'$'s with coefficients in $K$ whose degree (the degree of the numerator minus the degree of the denominator) is $2$ (resp. $3$).  Therefore, since $\sigma$ changes the sign of each $B_{i}'$ as well as each of the $A_{i}$'s and $B_{i}$'s, we see that $\sigma$ fixes the $x$-coordinate of $Q$ while changing the sign of the $y$-coordinate, so $\sigma(Q) = -Q$.  Thus, $\rho_{2}(\sigma) = -1 \in \SL(T_{2}(E))$, and we have proven the statement of the proposition for the degree-$3$ case.  In particular, we see from Proposition \ref{prop 4-torsion}(c) that if we choose $\sigma \in G_{K}$ to be an automorphism whose image under the restriction map to $G_{K}^{\unr} \cong \widehat{P}_{3}$ is the central element $\Sigma \in \widehat{P}_{3}$, then $\sigma$ acts on $K(E[8])$ by changing the sign of each of the $A_{i}$'s and $B_{i}$'s.

In the degree-$4$ case, similarly let $\sigma \in \Gal(\bar{K} / K(E[2]))$ be an automorphism whose image under the restriction map to $\Gal(K^{\unr} / K(E[2])) \cong \widehat{P}_{3}$ is the central element $\Sigma \in \widehat{P}_{3}$ as defined in the statement of Lemma \ref{lemma exponent 2 abelian}(a).  Then it follows from what was remarked at the end of the last paragraph that $\sigma$ again acts on $K(E[8])$ by changing the sign of each of the $A_{i}$'s and $B_{i}$'s, so to prove the theorem for this case it suffices to show that $\rho_{2}(\sigma) = -1 \in \SL(T_{2}(E))$.  But this is given by Proposition \ref{prop 4-torsion}(c).

\end{proof}

The author is grateful to Yuri Zarhin for his guidance in supervising the author's investigation of $8$-torsion of elliptic curves during his time at The Pennsylvania State University.  The author would also like to thank the referee for pointing out a gap in an earlier version of this work.

\bibliographystyle{plain}
\bibliography{bibfile}

\end{document}